\documentclass[11pt,a4paper]{article}
\usepackage[top=2cm, bottom=2cm, left=2.5cm, right=2.5cm]{geometry}
\usepackage{amsmath}

\usepackage{mathrsfs}

\usepackage[francais, english]{babel}
\usepackage[utf8]{inputenc}
\usepackage[T1]{fontenc}
\usepackage{amsfonts, amssymb, amsthm}
\usepackage{dsfont}
\usepackage{thmtools}
\usepackage{mathtools}
\usepackage{amsmath}

\usepackage{multirow}
\usepackage{graphicx}
\usepackage{caption, subcaption}
\usepackage{float}
\usepackage{enumitem}

\usepackage{pict2e}

\usepackage{mdframed}

\usepackage[all]{xy}

\usepackage{tikz}
\usetikzlibrary{cd} %Diagrammes Commutatifs
\usetikzlibrary{decorations} %Perturbation Aléatoires et autres décorations
\usetikzlibrary{decorations.pathmorphing}
\usetikzlibrary{decorations.pathreplacing}
\usepackage{tkz-tab} %Tableaux de signes et variations

\usepackage{multicol} %Environnement Multicolones

\mathtoolsset{showonlyrefs} %Numerotation Equations ssi Citee 

\usepackage{relsize,exscale}

\usepackage{hyperref}
\usepackage{url}
\usepackage{xurl}

 {\begin{list}{}%
         {\setlength{\leftmargin}{#1}}%
         \item[]%
 }
 {\end{list}}

\allowdisplaybreaks %autorisation de casser les environements align pour changer de page

\usepackage{import}

\makeatletter%
\@ifclassloaded{book}%
  {

}%
  {}%
\makeatother%
\newcounter{TheoremA}

\newtheorem{ThA}[TheoremA]{Theorem}

\newtheorem*{Cor*}{Corollary}

\newtheorem*{Conj*}{Conjecture} %Conjecture non numérotée pour l'intro
\newtheorem*{Question*}{Question} %Conjecture non numérotée pour l'intro

\makeatletter%
\@ifclassloaded{book}%
  {%
  	\newtheorem{Th}{Theorem}[chapter] %Si book alors recommencer à chaque chapitre
  }{%
  	\newtheorem{Th}{Theorem}[section] %Si PAS book alors le compteur recommence à chaque section (la commande que j'avais avant)
  }%
\makeatother%
\newtheorem{Lemma}[Th]{Lemma} %Même compteur que Theorem
\newtheorem{Prop}[Th]{Proposition}

\theoremstyle{definition}
\newtheorem{Def}[Th]{Definition}

\theoremstyle{remark}
\newtheorem{Remark}[Th]{Remark}

\newcounter{fig}

\newcommand{\R}{\mathbb{R}}
\newcommand{\C}{\mathbb{C}}
\newcommand{\Z}{\mathbb{Z}}
\newcommand{\N}{\mathbb{N}}
\newcommand{\scal}[2]{\left\langle#1,#2\right\rangle}
\DeclareMathOperator{\Tr}{Tr}

\DeclareMathOperator{\Rea}{Re} %Partie Réelle Complexe

\newcommand{\abs}[1]{\left\vert #1 \right\vert}

\newcommand{\Hspace}{\mathcal{H}}
\newcommand{\cH}{\mathcal{H}}

\DeclareMathOperator{\PSH}{PSH}

\DeclareMathOperator{\FS}{FS}
\DeclareMathOperator{\Hilb}{Hilb}

\DeclareMathOperator{\Amp}{Amp}
\DeclareMathOperator{\Herm}{Herm}

\DeclareMathOperator{\GL}{GL}

\DeclareMathOperator{\U}{U}

\usepackage{listings}
\usepackage{numprint}

\usepackage{setspace}
\usepackage{calc}

\title{Quantizing Geodesics in Kähler and Sasaki Geometry}

\author{Gilles Courtois, Eleonora Di Nezza, Thomas Franzinetti}
\date{}
\begin{document}

\maketitle
\abstract{
The space of Kähler potentials can be quantized through the classical Fubini–Study map, relating infinite-dimensional geometric structures to finite-dimensional symmetric spaces. We prove (exactly) when the Fubini–Study image of a geodesic line in the space of positive definite Hermitian matrices gives rise to a quasi-geodesic in the space of Kähler potentials. Furthermore, we introduce a quantization procedure for geodesics between potentials on normal Kähler varieties and show how this construction extends to the Sasaki setting.}
\vspace{0.3cm}


\section{Introduction}
Let $X$ be a compact K\"ahler manifold and $\omega$ be a Kähler form of complex dimension $n$. The space of Kähler potentials:
\begin{equation}
\mathcal{H}_\omega = \left\lbrace u\in\mathcal{C}^{\infty}(X, \R) \, \vert \, \omega_u := \omega + dd^c u > 0 \right\rbrace
\end{equation}
plays a central role in Kähler geometry. It is an infinite-dimensional manifold whose tangent space at each point can be identified with $C^\infty(X)$. At each $u\in \mathcal{H}_\omega$, the Mabuchi $L^2$-metric \cite{mabuchi_symplectic_1986} on the tangent space at $u$ is
\begin{equation}
\scal{\psi_1}{\psi_2}_u = \frac{1}{V }\int_X \psi_1 \psi_2 \,\omega_u^n \qquad \psi_1, \psi_2\in T_{u}\mathcal{H}_\omega \simeq \mathcal{C}^\infty(X),
\end{equation}
where $V:=\int_X \omega^n>0$.

%Here, $d^c := \frac{i}{2}(\overline{\partial} - \partial)$ so that $dd^c = i \partial \overline{\partial}$. 

 More generally, one may equip each tangent space with $L^p$-Finsler structures
\begin{equation}
\|\psi\|_u = \left (\frac{1}{V }\int_X |\psi|^p \omega_u^n\right)^{1/p} \qquad \psi \in T_{u}\mathcal{H}_\omega \simeq \mathcal{C}^\infty(X),
\end{equation}
leading to a family of distances $d_p$ ($p\geq 1$), and a well defined notion of geodesics between two elements of $\mathcal{H}_\omega$.  In fact, $d_p$ is defined as the length pseudometric associated to the Finsler $L^p$-norm, but one nontrivial point is that it is a genuine metric (i.e. separates points). Chen and his collaborators worked intensively in this direction \cite{chen_space_2000, calabi_space_2002, Chen_Tian_2008, Chen_Space_2009, Chen_Sun_2009} proving in particular that this infinite dimensional space is a path metric space (the metric is the infimum of the length of paths) with $\mathcal{C}^{1,\overline{1}}$ geodesics. 
Later Chu-Tosatti-Weinkove \cite{CTW18} has established that geodesics are $\mathcal{C}^{1,1}$-regular, which is known to be the optimal regularity thanks to examples of Darvas, Lempert, Vivas \cite{darvas2014morse, darvas_weak_2012, lempert2013geodesics}. In particular, geodesics in general leave the space $\cH_\omega$; we will refer to those as \emph{weak geodesics}.

As shown by Chen \cite{chen_space_2000} and \cite{Dar15} the distance $d_p$ between two smooth potentials $u_0,u_1$ is realized by the geodesic $u_t$:
\begin{equation}
	\label{eq: geodesic distance intro}
	d_p(u_0,u_1)^p= \int_X |\dot{u}_t|^p (\omega+dd^c u_t)^n, \; \forall t \in [0,1]. 
\end{equation}

The understanding of the geometry and topology of $\mathcal{H}_\omega$ is crucial for the study of special metrics in K\"ahler geometry (such as K\"ahler-Einstein, constant scalar curvature, extremal metrics...). A powerful method to study this infinite-dimensional space is the so called \emph{Kähler quantization}, which is a process to approximate the infinite dimensional space $\cH_\omega$ and its geometry by finite dimensional negatively curved manifolds. Such a process has been at the core of intense research over the past decades \cite{donaldson2001scalar, phong2005monge, song2010bergman, Chen_Sun_2009, berndtsson2013probability, boucksom2018variational, darvas2020quantization}.
%More precisely, in \cite{phong2005monge, berndtsson2013probability}, the authors establish a quantization process of geodesics between Kähler potentials.

%\subsection{About the Fubini-Study Map: $\FS_k$}

The "quantization process" requires the Kähler manifold to be projective. More precisely, we consider $(X,\omega)$ a compact Kähler manifold which carries a holomorphic line bundle $L\rightarrow X$ with a Hermitian metric whose curvature is $ \omega$. In this setting, $X$ comes with a family of embeddings into projective spaces:
$$
\Phi_k : X \hookrightarrow \mathbb{P}\left( H^0(X, L^k) \right),
$$
where $H^0(X,L^k)$ is the finite dimensional space of holomorphic sections of the $k$-th tensor power of $L$. By construction $\frac{1}{k}\Phi_k^*\omega_{FS}$ lies in the same cohomology class as ${\omega}$, hence $\frac{1}{k}\Phi_k^*\omega_{FS}=\omega+ dd^c u$, for some $u$ in $\cH_\omega$. Such smooth function is in the image of the Fubini-Study map
$$
\FS_k : \Herm^{++}(H^0(X,L^k)) \longrightarrow \cH_\omega.
$$
We refer to Section \ref{Sec:DefFSHilbk} for a definition. Here

$$\hat{\cH}(L^k) := \Herm^{++}(H^0(X,L^k)) = \GL(N_k,\C) / \U(N_k,\C)$$ is the symmetric Riemannian manifold which can be interpreted as the set of Hermitian scalar products on the (finite $N_k$-dimensional) space of holomorphic sections $H^0(X,L^k)$.
\medskip

%If no confusion may arise, let us write $\cH = \cH_\omega$, l
%[Corollary \ref{Cor:FSkInjective}]
%Let $d$ (usually denoted by $d_2$ in the literature) be the distance induced by the Mabuchi metric on $\cH_\omega$ and 
Let $\hat{d}$ be the distance on $\hat{\cH}(L^k)$.
Our first result states the following:
\begin{ThA}
\label{Th:INTROFSkLipschitz}
The map $\FS_k  : (\hat{\cH}(L^k) , \hat{d}) \rightarrow (\cH_\omega, d_p)$ is Lipschitz. More precisely,
$$
d_p(\FS_k(H_1), \FS_k(H_2)) \leq C_k \hat{d}(H_1, H_2), \quad \forall H_1, H_2 \in \hat{\cH}(L^k).$$
Here, $C_k$ is an explicit constant depending only on $k, X$, $\omega$ and $p$.
\end{ThA}

More generally, given a geodesic line $\R \ni  t\mapsto H(t)$ in $\hat{\cH}(L^k)$, it is natural to investigate how closely the paths
$
\R \ni t\longmapsto \FS_k(H(t))
$
approximate geodesics.
It is known that $\FS_k(H(t))\in \cH_\omega$ for any fixed $t$ and that for any geodesic \emph{segment} $[0,1] \ni t \mapsto H(t) \in \hat{\cH}(L^k)$, the path $\FS_k(H(t))$ lies below the geodesic between $\FS_k(H_0)$ and $\FS_k(H_1)$ \cite{phong2005monge} (i.e. $\FS_k(H(t))$ is a \emph{subgeodesic}). Moreover, such subgeodesic converges in the $\mathcal{C}^0(X)$-topology to the geodesic segment in $\cH_\omega$.

 It is well known that a geodesic line in $\hat{\cH}(L^k)$ passing through $H_0=PP^*$, $P\in GL(N,\mathbb C)$, reads as $H(t) = P \exp (tA) P^*$. We denote by $\lambda_1\leq \cdots \leq \lambda_{N_k}$ the eigenvalues of $A$.
 Our next result concerns the globally defined path
$
\R \ni t\longmapsto \FS_k(H(t))$:

%\[Proposition ref{Prop:ImageOfGeodesicIsQuasiGeodesic_SameSignEigenvalues}]
\begin{ThA}
\label{Prop:INTROImageOfGeodesicIsQuasiGeodesic_SameSignEigenvalues}
Let $\R \ni t\mapsto H(t)$ be a geodesic in $\hat{\cH}(L^k)$, parametrized with constant speed. Assume $A$ is either definite or $\lambda_1<0$ and $\lambda_{N_k}>0$.
%Assume that the eigenvalue of the Hermitian matrix $\frac{d}{dt}\vert_{t=0} H(t)$ all have the same sign.
Then the path
$
\R \ni t \longmapsto \FS_k(H(t))
$
is a quasi-geodesic in $\cH_\omega$.
\end{ThA}

Recall that a quasi-geodesic in a metric space $(\mathcal{M}, d_{\mathcal{M}})$ is a path $\R \ni t \mapsto \gamma_t \in \mathcal{M}$ satisfying:
$$
\forall t,u \in \R, \quad C^{-1}\abs{t-u} - K \leq d_{\mathcal{M}}(\gamma_t, \gamma_u) \leq C\abs{t-u} + K,
$$
for some $C\geq 1$ and $K\geq 0$.

We also emphasize that the assumptions on the eigenvalues of $A$ are necessary as the example in Remark \ref{ex} shows.
\smallskip

It is also worth mentioning that Phong–Sturm \cite{PhongSturm2009}, Ross–Witt Nyström \cite{RWN_geo}, and Finski \cite{Finski} (in the more general big setting) show that, given a sequence (in $k$) of geodesic rays on $H^0(X, L^k)$ constructed from carefully chosen norms on $H^0(X, L^k)$  and appropriate speeds, the resulting sequence of quasi-geodesics converges, as $k\rightarrow +\infty$, to a weak \emph{geodesic ray}. An interesting open question is whether, and under what conditions, the quasi-geodesic $\R \ni t \longmapsto \FS_k(H(t))$ (defined on the entire real line) converges to a weak g\emph{eodesic line}.

\bigskip
%\subsubsection{Quantization of Geodesics on compact Kähler varieties}

Motivated by the study of the space of Kähler potentials on \emph{singular varieties} \cite{boucksom_monge-amp`ere_2008,darvas2017metric,di2018geometry}, we consider  the space of potentials with respect to a \emph{big} form. More precisely, as before we work on $(X,\omega)$ a compact smooth Kähler manifold  and we fix another smooth $(1,1)$-form $\theta$, which we assume to be real, closed, semi-positive and big. We say that $\theta$ is \emph{big} if ${\theta}$ can be represented by a strictly positive current, or in other words, if there exists $u\in \PSH(X, \theta)$ such that $\theta+dd^c u \geq \varepsilon \omega$, for some $\varepsilon>0$. When $\theta$ is semi-positive this simply amounts to $\int_X \theta^n>0$.\\
We recall that $\PSH(X, \theta)$ is the set of $\theta$-plurisubharmonic functions, i.e. the set of upper semicontinuous functions
$\{u:X\to \mathbb{R}\cup\{-\infty\} $ which locally write as the sum of a smooth and a plurisubharmonic function and such that $\theta_u:=\theta+dd^c u \geq 0$ in the weak sense of currents.
%We say a function is upper semicontinuous ({u.s.c.} for short) if for all $y\in X$  
%$ \limsup_{x\to y} u(x)\leq u(y). $
We set
$$\cH_\theta :=\{ u\in \PSH (X, \theta)\cap L^\infty(X) \cap C^\infty(\Amp(\theta )) \;\, : \;\,\theta_u:=\theta+dd^c u>0 \; {\rm{in}} \; \Amp(\theta )\}.$$
Here $\Amp( \theta)$ is the open set of points $x \in X$ such that there exists $\phi \in \PSH(X,\theta)$ which is smooth near $x$ and with analytic singularities (that is it can be locally written: $\phi = c\log \sum_j \abs{f_j}^2 +f$ where the $f_j$'s are holomorphic, $f$ is smooth and $c$ is a constant). We also mention that $\Amp( \theta)$ is a Zariski dense open subset of $X$ and it only depends on the cohomology class of $\{\theta\}$.
Moreover, it is shown in \cite[Theorem 3.17]{boucksom2004divisorial} that $\theta$ is Kähler if and only if $\Amp( \theta ) = X$.
\medskip

In \cite{di2018geometry}, the authors extended Mabuchi geometry to $\cH_\theta$, defining a distance and a notion of geodesics between potentials in $\cH_\theta$. Here we provide a quantization result in this more general setting generalizing Berndtsson's approximation of geodesics:

\begin{ThA}\label{Th:INTROQuantizationGeodesicsSemiPositiveEnonce1}
Let $u_0, u_1 \in \Hspace_\theta$ and $[0,1]\ni t\rightarrow u_t$ be the weak geodesic joining them. 
Then there exists a sequence of paths $(\hat{u}_{t,\ell})_{\ell\in\N}$ in the image of $\frac{1}{\ell}\FS_k$, where $k$ depends on $\ell$, such that for any compact $K\subset \Amp(\theta)$, $\hat{u}_{t,\ell}$ converges uniformly to $u_t$ in $K\times[0,1]$.
\end{ThA}
We point out that the recent preprint \cite[Theorem 1.5 and Proposition 5.2]{Finski} provides an alternative proof of the convergence stated in Theorem C in the general big (not necessarily semi-positive) setting. As explained in \cite[Remark 1.6]{Finski}, the uniform convergence was anticipated. Our arguments yields uniform convergence directly in the semi-positive case.

\bigskip
%\subsection{The Space of Sasaki Potentials}
Building on the above result, we develop a corresponding quantization process in the Sasaki setting. Recall that a Sasaki manifold $(M,g)$ is an odd dimensional compact Riemannian manifold whose metric cone $C(M) = (\R^*_+ \times M, dr^2+r^2 g)$ is Kähler. In particular, $M$, which can be identified with the link $\left\lbrace r = 1 \right\rbrace \subset C(M)$, is a contact manifold with contact form $\eta = 2d^c \log(r)$. It defines a contact bundle $\ker \eta$ on which $\frac{1}{2}d\eta$ is a transverse Kähler form.  Moreover, any Sasaki manifold is endowed with a special vector field: the Reeb vector field $\xi$ which is the restriction of $J(r\partial_r)$ to the link $\left\lbrace r = 1 \right\rbrace$. Here $J$ denotes the complex structure on the Kähler metric cone $C(M)$. The restriction $\Phi$ of $J$ to the transverse distribution $\ker\eta$ is called a transverse complex structure. Using $\xi, \eta$ and $\Phi$, one can define a Riemannian metric: $g = \eta \otimes \eta + \frac{1}{2} d\eta (\cdot, \Phi \cdot)$. We call $\mathcal{S} := (\xi, \eta, \Phi, g)$ a \emph{Sasaki structure}.

By analogy with Kähler potentials, we consider
the space of potentials on a Sasaki manifold $(M,\mathcal{S})$:
$$
\mathcal{H}(M, \mathcal{S}) = \left\lbrace \phi\in\mathcal{C}^{\infty}_B(M),\; d\eta + dd^c\phi > 0 \right\rbrace,
$$
where $\mathcal{C}^{\infty}_B(M)$ is the space of smooth \emph{basic} functions (ie smooth functions which are invariant under the Reeb flow). Analogously to the Kähler setting, the $d^c$-operator, acting on basic functions, is defined as $\frac{i}{2}(\overline{\partial}_B - \partial_B)$ where the $\partial_B$ and $\overline{\partial}_B$ operator are defined in \cite{boyer_sasakian_2007}. We omit the subscript $B$ for simplicity. This infinite dimensional space, whose tangent space at any $u\in\mathcal{H}(M, \mathcal{S})$ is identified with $\mathcal{C}_B^\infty(M)$, is endowed with a Riemannian structure, analogue of the Mabuchi metric \cite{guan_geodesic_nodate, guan_regularity_2009}:
$$
\scal{\psi_1}{\psi_2}_u = \frac{1}{\int_M \eta \wedge d\eta^n}\int_M (\psi_1 \psi_2) \; \eta\wedge (d\eta + dd^c u)^n.
$$
Here and in the sequel, $2n+1$ is the real dimension of $M$.

Guan and Zhang \cite{guan_geodesic_nodate, guan_regularity_2009} proved the existence of $\mathcal{C}^{1,\overline{1}}$ geodesics using a Monge-Ampère type re-formulation for the geodesic equation. They also showed that the Riemannian structure on the tangent space of $\mathcal{H}(M, \mathcal{S})$ induces a distance  on $\mathcal{H}(M, \mathcal{S})$.
%$$
%d(u_0, u_1) := \inf\left\lbrace   \int_0^1 \sqrt{\scal{\dot{u_t}}{\dot{u_t}}_{u_t}}dt \, \bigg\vert \, t\mapsto u_t \text{ smooth path from } u_0 \text{ to } u_1 \right\rbrace.
%$$
\smallskip

We will make use of the results in the K\"ahler setting to study quantization for regular and quasi-regular Sasaki manifolds. Indeed, recall that if a Sasaki manifold is quasi-regular, then the space of orbits of the Reeb vector field is a Kähler orbifold. If it is regular, then the space of orbits is a Kähler manifold.

If $(M,\mathcal{S})$ is a regular or quasi-regular Sasaki manifold, we let $\pi : M \rightarrow X$ be the projection onto the space of orbits under the Reeb vector field.
We let $M^{\text{reg}} := \pi^{-1}(X^{\text{reg}})$ be the set of points in $M$ which lie in a regular orbit under the Reeb vector field (if $M$ is regular, then $M^{\text{reg}} = M$). In addition, we consider $\hat{\cH}(L^{k})$, the space of Hermitian metrics on the finite dimensional space of holomorphic sections of a certain power of a line bundle $L$ over a resolution of singularities of $X$ (see Section \ref{sec:QuantizationSasaki}).
In this setting we have:

\begin{ThA}
\label{Th:INTROQuantizGeodesicSasaki}
Let $u_0, u_1 \in \cH(M,\mathcal{S})$ and let $u_t$ be the weak geodesic joining $u_0$ and $u_1$. There exists a sequence of paths $(\hat{u}_{t,\ell})_{\ell\in\N}$, which are images of geodesics in $\hat{\cH}(L^k) $, such that for any compact set $K\subset M^{\text{reg}}$, $\hat{u}_{t,\ell}$ converges uniformly to $u_t$ in $K\times[0,1]$ .
\end{ThA}

%The above statement is intentionally left vague and we refer to Theorem ~\ref{Th:QuantizationGeodesicsSemiPositive} and Theorem \ref{Th:QuantizGeodesicSasaki} for a precise statement, for an explicit expression of the sequence of path $(\hat{u}_{t,l})_{l\in\N}$ and for an explicit choice of both the line bundle $L$ and the power $k(l)$.

%Theorem \ref{Th:INTROQuantizGeodesicSasaki} is an almost straightforward consequence of a much more general result which we explain below.

\paragraph{Organization of the paper.}  Section \ref{Sec:QuantizationKahler} recalls the geometry of Hermitian symmetric spaces and the definition of the Fubini–Study map, and proves Theorems \ref{Th:INTROFSkLipschitz} and \ref{Prop:INTROImageOfGeodesicIsQuasiGeodesic_SameSignEigenvalues}.
Section \ref{Sec:PotentialsWithSemiPositiveForm} establishes the quantization result in the semi-positive and big setting (Theorem \ref{Th:INTROQuantizationGeodesicsSemiPositiveEnonce1} ). Section \ref{sec:QuantizationSasaki} develops the Sasaki framework and proves Theorem \ref{Th:INTROQuantizGeodesicSasaki}.

\section*{Acknowledgements}
%The first author is supported by the project SiGMA ANR-22-ERCS-0004-02. 
The second author is supported by the ERC grant SiGMA (No. 101125012, PI: Eleonora Di Nezza). She also gratefully acknowledges the University of Melbourne (MATRIX campus) for providing a welcoming and stimulating environment during the research program ``Analytic and Geometric Methods on Complex Manifolds" in which this project was completed.
Moreover, we thank S\'ebastien Boucksom, Tamas Darvas, Siarhei Finski and Mattias Jonsson for helpful discussions and comments on a previous draft of the note.

%Part of this material is based upon work done while the first author was supported by the National Science Foundation under Grant No. DMS-1928930, while the author was in

\section{The Fubini-Study map $\FS_k$}
\label{Sec:QuantizationKahler}

After a brief review over the geometry of the set of Hermitian matrices, we will recall the definitions of the main two maps $\FS_k$ and $\Hilb_k$ which are at the core of Kähler quantization.

\subsubsection*{The space of Hermitian Matrices} 
 
We furnish the space of Hermitian matrices with a structure of symmetric Riemannian manifold. We are particularly interested in the explicit description of geodesics.

\medskip

Let $\Herm^{++}(N,\C)$ be the space of definite positive Hermitian matrices of size $N\times N$.
We furnish $\Herm^{++}(N,\C)$ with a Riemannian metric:
\begin{equation}
\label{eq:DefRiemMetricOnHerm}
\scal{H_1}{H_2}_{H}: = \frac{1}{N} \Tr (H^{-1}H_1 H^{-1}H_2),
\end{equation}
where $H \in \Herm^{++}(N,\C)$ and $H_1, H_2 \in T_H{\Herm^{++}(N,\C)} \simeq \Herm(N,\C)$. The scaling factor $\frac{1}{N}$ will appear to be nice in our setting since the dimension of the matrices that we will consider will be bigger and bigger.

\medskip
Let us recall that $\text{Herm}^{++}(N,\mathbb C)$ is a symmetric space of non positive curvature, \cite{Hel}. 
For $P\in GL(N,\mathbb C)$, we denote $P^{*}$ the adjoint of $P$. The action of $GL(N, \mathbb C)$ on $\text{Herm}^{++}(N,\mathbb C)$ defined for $P\in GL(N,\mathbb C)$ and $H\in \text{Herm}^{++}(N,\mathbb C)$ by
$P. H := P H P^{*}$ is a transitive isometric action. In particular every $H\in \text{Herm}^{++}(N,\mathbb C)$ can be written as $P.Id=PP^*$ for some $P\in GL(N,\mathbb C)$. Hence, $\text{Herm}^{++}(N, \mathbb C)$ may be identified with
$GL(N, \mathbb C) \slash U(N,\mathbb C)$, where $U(N, \mathbb C)$ is the stabilizer of $Id$.
A path $t\mapsto H(t) \in \Herm^{++}(N,\C)$ is a geodesic with respect to the metric \eqref{eq:DefRiemMetricOnHerm} if and only if $\ddot{H} = \dot{H}H^{-1} \dot{H}$, or equivalently:
$
\frac{d}{dt}\left( H^{-1} \dot{H} \right) = 0.
$

\noindent As a consequence, geodesics $H(t)$ in $\Herm^{++}(N,\C)$ such that $H(0) = H_0$
all are of the form:
\begin{equation}
\label{paramgeod}
\R \ni t \mapsto H(t) = P \exp(tA) P^{*},
\end{equation}
where $H_0= P P^{*}$ with $P\in GL(N,\mathbb C)$ and $P A P^{*}  \in \Herm(N,\C)$ is the initial direction. Notice that the above geodesics have constant speed equal to $\sqrt{\frac{1}{N}\text{Tr} (A^2)}$,
in particular, for any $s < t \in \R$, the riemannian distance between $H(s)$ and $H(t)$ is equal to 
\begin{equation}
\label{distancealonggeodesic}
\hat{d}(H(s), H(t)) = \sqrt{\frac{1}{N} \Tr \left( A^2 \right)} \abs{t - s}.
\end{equation}

\subsubsection*{The Kähler Quantization}
\label{Sec:DefFSHilbk}

Consider $(L,h_0)\rightarrow X$ a positive line bundle whose curvature is $\omega$. Observe that, given any $u \in \Hspace_\omega$, $h_0 e^{-u}$ is an other positive hermitian metric on $L$. By abuse of notation we will still denote by $h_0$ the induced hermitian metric on the space of holomorphic sections $H^0(X,L^k)$, defined as
$$h_0(s, s'):= \int_X h_0(s(x), s'(x)) \omega^n, \qquad s,s'\in H^0(X,L^k).$$ The space $H^0(X,L^k)$ is a finite dimensional space, we write $N_k$ its dimension. Let $\hat{\Hspace}(L^k)$ be the space of Hermitian metrics on $H^0(X,L^k)$. The space $\hat{\Hspace}(L^k)$ can be identified with $\Herm^{++}(N_k, \C)$. Indeed, given $\{s_i\}$ an orthonormal basis for $h_0$, any $h\in\hat{\Hspace}(L^k)$ is associated with a unique $H\in \Herm^{++}(N_k, \C)$ so that $H_{ij}=h(s_i, s_j)$. In the following, using the above identification, we will talk about the $H$-orthonormal basis of sections, instead of $h$-orthonormal.

The goal here is to use the finite dimensional spaces $\hat{\cH}(L^k)$ to approximate the infinite dimensional space $\cH_\omega$.
Such a quantization process involves two maps:
$$
\Hilb_k : \cH_\omega \rightarrow \hat{\Hspace}(L^k), \qquad \FS_k : \hat{\Hspace}(L^k)\rightarrow \cH_\omega.
$$

Let us start with the definition of the map $\Hilb_k : \Hspace_\omega \rightarrow \hat{\Hspace}(L^k)$ associating to each $u \in \Hspace_\omega$ the Hermitian metric:
$$
\forall s \in H^0(X,L^k), \quad \Hilb_k(u)(s,s) := \int_X \abs{s}_{(h_0 e^{-u})^k}^2 \omega^n = \int_X \abs{s}_{h_0^k}^2 e^{-k u} \omega^n.
$$

\begin{Remark}
\label{Rk:VolumeFormHilb}
Note that the volume form we use here is not the same as in the classical definition \cite[p347]{donaldson2005scalar} of the $\Hilb_k$ map.
\end{Remark}

In the opposite direction, given $H$ a Hermitian metric on $H^0(X,L^k)$, we consider $\left\lbrace s_j \right\rbrace$ a $H$-orthonormal basis and we set:
\begin{equation}
\label{eq:DefFSk}
\FS_k (H) = \frac{1}{k}\log\left( \sum_j \abs{s_j}_{h_0^k}^2 \right) \in \cH_\omega.
\end{equation}

We recall a fundamental result due to Berndtsson \cite[Theorem 1.2]{berndtsson2013probability}:

\begin{Th}
\label{Th:GeometricQuantizationKahler}
%Let $L\rightarrow (X,\omega)$ be a Hermitian line bundle of curvature $\frac{1}{2\pi}\omega$. 
Let $u_0, u_1 \in \Hspace_\omega$, and $[0,1] \ni t \mapsto u_t$ be the weak geodesic joining $u_0$ and $u_1$. We write $H_i = \Hilb_k(u_i)$ for $i=0,1$ and we let $H(t)$ be the geodesic in $\hat{\Hspace}(L^k)$ from $H_0$ to $H_1$. Then there is a constant $C>0$ such that:
$$
\sup_{(x,t) \in X \times [0,1]} \abs{\FS_k(H(t))(x) - u_t(x)} \leq C \frac{\log k }{k}.
$$
In particular, the weak geodesic $u_t$ can be approximated in the $\mathcal{C}^0(X)$-topology by the quantized path $\FS_k(H(t))$.
\end{Th}

As mentioned in Remark \ref{Rk:VolumeFormHilb}, our definition of the map $\Hilb_k$ is not the classical one. It follows from \cite{berndtsson2013probability} that we can choose any volume form in the definition of $\Hilb_k$ so that the above classical approximation theorem is still valid. We refer to \cite[Theorem~1.2]{darvas2020quantization} for more details.

\subsubsection*{Quantized geodesic lines are quasi-geodesics}
\label{Sec:ImageByFSkOfGeodesics}

%It is natural to wonder what is the image of a geodesic by the map $\FS_k$. Let us investigate to what extent such a path is close to a geodesic.

%Here and in the sequel, $(L,h)\rightarrow (X,\omega)$ is a Hermitian line bundle whose curvature is $\frac{1}{2\pi}\omega$.

In view of Theorem \ref{Th:GeometricQuantizationKahler}, the image by $\FS_k$ of the geodesic segment between $\Hilb_k (u_0)$ and $\Hilb_k (u_1)$ in $\Herm^{++}(N_k,\C)$ is close, in $\mathcal{C}^0$-norm, to the geodesic in $\cH_\omega$ between the two potentials $u_0$ and $u_1$, at least for $k$ large enough. Notice that this holds for \emph{bounded geodesic segment}. 
%Geodesics in $\Herm^{++}(N_k,\C)$ are well defined as geodesic rays $\R \rightarrow \Herm^{++}(N_k,\C)$. 
We now study the image by the map $FS_k$ of \emph{complete geodesic lines} 
$\R \rightarrow \Herm^{++}(N_k,\C)$ in \eqref{paramgeod}, keeping in mind
the same question as above: are they close to geodesics lines in $\Hspace_\omega$? We don't answer this question but we prove that $FS_k$ maps geodesic lines in $\Herm^{++}(N_k,\C)$ onto quasi-geodesics in $\Hspace_\omega$. Let us recall the notion of quasi-geodesics:

\begin{Def}
\label{Def:QuasiGeodesic}
A \emph{quasi-geodesic} in a metric space $(\mathcal{M},d_\mathcal{M})$ is a path $I \ni t\mapsto \gamma_t \in \mathcal{M}$ such that there exist constants $C, K >0$ such that:
$$
\forall t,s \in I, \quad C^{-1} \abs{t-s} - K \leq d_\mathcal{M}(\gamma_t, \gamma_s) \leq C\abs{t-s} + K.
$$
Here, $I\subset \R$ is an interval.
\end{Def}

In what follows we will need the following two lemmas. The first one provides a comparison between the Mabuchi distance and some explicit analytic expression.

\begin{Lemma}[{\cite[Theorem 3.32]{darvas735geometric}}]
\label{Lemma:DarvasDistanceL2normComparison}
There exists $C=C(p)>0$ such that for any $u_0, u_1 \in \Hspace_\omega$ we have:
$$
\frac{1}{C}d_p(u_0,u_1)^p \leq \frac{1}{V}\int_X \abs{u_0 - u_1}^p \left( \omega_{u_0}^n + \omega_{u_1}^n \right) \leq C d_p(u_0,u_1)^p.
$$
\end{Lemma}

The second lemma describes the variation of the potential $FS_k (H(t))$ along a geodesic $H(t)$ of $\hat{\mathcal H} (L^k) $ identified with $ \Herm^{++}(N_k, \mathbb C)$.
\begin{Lemma}\label{u(t)-u(s)-lemma}
Let $A\in \Herm (N_k, \mathbb C)$ and
$H_0 = P P ^* \in \Herm^{++}(N_k, \mathbb C)$ %\textcolor{red}{can we always write $H_0$ like that?} \textcolor{green}{c'est expliqu\'e au d\'ebut de la section 2.1: $GL(N_, \mathbb C)$ agit transitivement sur $\Herm^{++}$ par $PHP^*$, et on applique \`a $H =  Id$.} 
and $H(t) := P \exp (tA) P^*$ a geodesic passing through $H(0) = H_0$. Then, for every $s \leq t$,
\begin{equation}\label{u(t)-u(s)}
\frac{\max(\lambda_1, -\lambda_{N_k})}{k} \abs{t-s} \leq 
\abs{\FS_k(H(t)) - \FS_k(H(s))} \leq \frac{\max(-\lambda_1, \lambda_{N_k})}{k} \abs{t -s},
\end{equation}
where $\lambda_1\leq \lambda_2\leq \cdots \leq \lambda_{N_k}$ are the eigenvalues of $A$.
\end{Lemma}
%Observe that since the action $P.H:=PHP^*$ is transitive, we can infer that any $H_0\in \Herm^{++} (N_k, \mathbb C)$ can be written as $PP^*$.
\begin{proof}
Let $\{s_j\}, j=1,...,N_k$, be a $H_0$-orthonormal basis of $\hat{\mathcal H} (L^k)$. We claim that $\{s_j e^{-\frac{\lambda_j t}{2}}\}$ is an orthonormal basis for $H(t)= P \exp (tA) P^{*}$. To see this, we first observe that $\{E_j\}$ is an orthonormal basis for $H_0=PP^*$ w.r.t. $h_0:=<\cdot,\cdot>_0$ (i.e. $<PP^\star E_j, E_k>_0=\delta_{jk}$) if and only if $\{e_j\}:=\{P^*E_j\}$ is an orthonormal basis for $<\cdot,\cdot>_0$.  \\
Now, let $\{e_j\}$ be an orthonormal basis for $<\cdot,\cdot>_0$ which diagonalize $A$, i.e. $e^{tA}e_j= e^{t\lambda_j}e_j$ then $\{e^{-{t\lambda_j}/{2}} E_j\}$ (where $E_j$ is defined so that $P^*E_j=e_j$) is an orthonormal basis for $Pe^{tA}P^*$. The above arguments applied to $E_j=s_j$ prove the claim.\\
Hence we have
\begin{equation}\label{FSkHt}
\FS_k(H(t)) = \frac{1}{k}\log\left( \sum_j \abs{s_j}_{h_0^k}^2 e^{-t\lambda_j} \right).
\end{equation}
\noindent
Let us denote
\begin{equation}\label{def-u}
u(t):= \FS_k(H(t)).
\end{equation}
We then have
$$
u'(t) = -\frac{1}{k}\frac{\sum_j \lambda_j \abs{s_j(x)}_{h_0^k}^2 e^{-t\lambda_j}}{\sum_j \abs{s_j(x)}_{h_0^k}^2 e^{-t\lambda_j}}.
$$
On the other hand, for any $1\leq j \leq N_k$, we have $\lambda_1 \leq \lambda_j \leq \lambda_{N_k}$, hence
$$
-\lambda_{N_k} e^{-t\lambda_j} \abs{s_j(x)}_{h_0^k}^2
	\leq -\lambda_{j} e^{-t\lambda_j} \abs{s_j(x)}_{h_0^k}^2
	\leq -\lambda_{1} e^{-t\lambda_j} \abs{s_j(x)}_{h_0^k}^2.
$$
Summing up the previous inequalities on $1\leq j \leq N_k$, we get
$$
-\frac{\lambda_{N_k}}{k} \leq u'(t) \leq -\frac{\lambda_1}{k},
$$
therefore, for any $s<t$, we have
$$
-\frac{\lambda_{N_k}}{k} \leq \frac{u(t) - u(s)}{t-s} \leq -\frac{\lambda_1}{k}.
$$
We deduce
\begin{equation}
\frac{\max(\lambda_1, -\lambda_{N_k})}{k} \abs{t-s} \leq \abs{u(t) - u(s)} \leq \frac{\max(-\lambda_1, \lambda_{N_k})}{k} \abs{t -s},
\end{equation}
which concludes the proof.
\end{proof}

\begin{Th}
\label{Cor:FSkInjective}
The map $\FS_k : (\hat{\cH}(L^k), \hat{d}) \rightarrow (\cH,d_p)$ is $C(k)$-Lipschitz with Lipschitz constant $C(k):= \frac{C_p}{k}\sqrt{ N_k}$, where $C_p$ is a positive constant which depends on $p$.
\end{Th}
\begin{Remark}
We remark that, since $N_k\sim k^n$ as $k\rightarrow +\infty$, the Lipschitz constant $C(k)$ can be chosen independently on $k$ when $n=2$. We also note that in the case of Riemann surfaces $FS_k$ is a contraction for $k>>1$.
\end{Remark}

\smallskip
We mention that Lempert \cite[Theorem 1.1]{Lempert21} proved that the map $\FS_k$ is injective.

\begin{proof}
Let $H_0 , H_1 \in \text{Herm}^{++}(N_k, \mathbb C)$. By \eqref{paramgeod}, the geodesic segment $H(t), t\in [0,1]$  joining $H_0$ and $H_1$ writes
$H(t)= P \exp (tA) P^{*}$ where 
$H_0 = P P^{*}$ and $A= \log \left( P^{-1} H_1 (P^{*})^{-1}\right)$. By \eqref{distancealonggeodesic}, the distance between $H_0$ and $H_1$ is equal to $\hat d (H_0, H_1) = \sqrt{\frac{1}{N_k} \text{Tr} (A^2)}$. \\
Denote by $u_i:=\FS_k(H_i)$. From Lemma \ref{Lemma:DarvasDistanceL2normComparison} and the right hand side of \eqref{u(t)-u(s)}
we get
\begin{align*}
d_p\left( \FS_k(H_0), \FS_k(H_1) \right)
&\leq \left(\frac{C_p}{V}\right)^{1/p}\left(\int_X \abs{u_0 - u_1}^p \left( \omega_{u_0}^n + \omega_{u_1}^n \right)\right)^{1/p} \leq C_p'  \, \frac{\max(\lambda_{N_k}, -\lambda_1)}{k}.
	\end{align*}
Observing that 
	$
\max(\lambda_{N_k}, -\lambda_1) = \max(\abs{\lambda_1}, \abs{\lambda_{N_k}}) \leq \sqrt{\Tr(A^2)} = \sqrt{N_k} \hat{d} (H_0, H_1),
$
we deduce from the previous inequality that 
$$d_p( \FS_k(H_0), \FS_k(H_1) )
\leq C_p'\frac{\sqrt{N_k}}{k} \hat{d}(H_0, H_1),$$
which concludes the proof of Theorem \ref{Cor:FSkInjective}.
\end{proof}
%As we said, the first inequality above is a consequence of the proof of Proposition \ref{Prop:ImageOfGeodesicIsQuasiGeodesic_SameSignEigenvalues} and of \eqref{eq:DistanceFormuleGeodesicCauchyInHermitianMatrices}. The 

\noindent
Let us consider $A \in \Herm(N_k,\C)$ such that $A \neq 0$ and $H(t) := P \exp (tA) P^{*}$ a geodesic ray in $\Herm ^{++}(N_k, \mathbb C)$ passing through $H_0 = P P^*$. Observe that, since $A\neq 0$, then $H(t)$ is not constant in time. The following result shows when 
$t \to \FS_k(H(t))$ is a quasi-geodesic in $\mathcal H $.

\begin{Th}
\label{Prop:ImageOfGeodesicIsQuasiGeodesic_SameSignEigenvalues}
Let $0 \neq A \in \Herm(N_k,\C)$ and $H_0 =P P^* \in \hat{\Hspace}(L^k)$. Assume $A$ is either definite or $\lambda_1<0$ and $\lambda_{N_k}>0$. Then, the image by $\FS_k$ of the geodesic $\R \ni t \mapsto H(t) := P\exp(t A)P^*$ is a quasi-geodesic in $\Hspace_\omega$. More precisely, there exist positive constants $C_k, D_k>0$ such that
$$
 C_k \hat{d}(H(t), H(s)) -D_k \leq d_p(\FS_k(H(t)),\FS_k(H(s)))   \leq C_p \frac{\sqrt{N_k}}{k} \hat{d}(H(t), H(s)).
$$
\end{Th}

\begin{Remark}\label{ex}
The assumption on the eigenvalues of $A$ is optimal as the following example shows. 
First, we observe that up to work with a large power of $L$ we can suppose that $L$ is very ample and that $H^0(X, L)$ has a basis $(s_1,...,s_N)$ so that $s_1,...,s_p$ have no common zeros for some $p<N$. Now, given $H(t)=diag(1,...,1,e^{t},...,e^{t})$, ie. $\lambda _1=\cdots:\lambda_p =0 <1= \lambda_{p+1} =\cdots=\lambda _N$, we can infer that
$$\FS_1(H(t))=\log\left(\sum_{i\le p} |s_i|^2+e^{-t}\sum_{i>p}|s_j|^2 \right) $$ decreases at $t$ goes to $+\infty$ to $\phi=\log \sum_{i\le p} |s_i|^2$. In particular $d_p(\FS_1(H(t)), \FS_1(H(0)))$ is bounded in $t$. On the other side, $\hat{d}(H(t),H(0))= \frac{1}{\sqrt{N}} t$. This implies that we can not find  uniform constants (in $t$) $A, B>0$ so that $\hat{d}(H(t),H(0)) \leq A d_p(\FS_1(H(t)),\FS_1(H(0)))+B$.
\smallskip

We thank Sebastien Boucksom and Mattias Jonsson for sharing this example.
\end{Remark}

\begin{proof}
By Theorem \ref{Cor:FSkInjective}, the map $FS_k$ is Lipschitz, thus it remains to show the left hand side inequality of the statement. Let us denote by
$ \lambda_1 \leq \cdots \leq   \lambda_{N_k}$ the eigenvalues of $A$. Observe that if $\lambda_1>0$ or $\lambda_{N_k}<0$, then the conclusion follows straightforwardly by Lemma \ref{u(t)-u(s)-lemma}. 
\smallskip
In what follows we consider the case $\lambda_1<0$ and $\lambda_{N_k}>0$. Let $\left\lbrace s_j \right\rbrace$ of $H^0(X,L^k)$ be a $H_0$-orthonormal basis.
As we already saw in the proof of Lemma \ref{u(t)-u(s)-lemma} we have
$$
u(t):=\FS_k(H(t)) = \frac{1}{k}\log\left( \sum_j \abs{s_j}_{h_0^k}^2 e^{-t\lambda_j} \right).
$$
To simplify notation, in the following we simply write $\abs{s_j}$ instead of $\abs{s_j}_{h_0^k}^2$. Now, for $\tau\in \R^+$
\begin{eqnarray*}
u'(\tau) &=& \frac{1}{k}\frac{\sum_{\lambda_j <0} |\lambda_j| \abs{s_j}^2 e^{-\tau\lambda_j} - \sum_{\lambda_i>0}|\lambda_i| \abs{s_i}^2 e^{-\tau\lambda_i} }{ \sum_{\lambda_\ell<0} \abs{s_\ell}^2 e^{-\tau \lambda_\ell} +  \sum_{\lambda_\ell\geq 0} \abs{s_\ell}^2 e^{-\tau \lambda_\ell} }:= \frac{1}{k} \frac{A(\tau)-B(\tau)}{C(\tau)+D(\tau)}.
\end{eqnarray*}
Observe that  $D(\tau)$ is bounded above independently of $\tau$ since $D(\tau) \leq \sum_\ell \abs{s_\ell}^2\leq C_1$. Also,
\begin{eqnarray*}
 A(\tau)-B(\tau)
 &=& \sum_{\lambda_j<0} e^{-\tau \lambda_j} \left( |\lambda_j| |s_j|^2- \sum_{\lambda_i>0}   |\lambda_i| |s_i|^2 e^{-\tau (\lambda_i+|\lambda_j|)} \right)\\
  &\geq & \sum_{\lambda_j<0} e^{-\tau \lambda_j} \left( |\lambda_j| |s_j|^2- h(\tau) \right),
 \end{eqnarray*}
 where $h:\R^+\rightarrow\R^+ $ smooth function going to zero as $\tau\rightarrow +\infty$.
%Observe that  $D(\tau)$ is bounded above independently of $\tau$ since $D(\tau) \leq \sum_\ell \abs{s_\ell}^2\leq C_1$. Hence 
%$$|u'(\tau)|^2 \geq \frac{1}{k^2} \frac{A^2(\tau)-2A(\tau) B(\tau)}{2C(\tau)^2+2C_1} \geq \frac{1}{k^2} %\frac{A^2(\tau)-2A(\tau) B(\tau)}{2C_1^2 \left( \sum_{\lambda_\ell<0 }e^{-2\tau \lambda_\ell}+1 \right)}  ,$$
%since $C(\tau) \leq C_1 \sum_{\lambda_\ell<0 }e^{-\tau \lambda_\ell}$. Also,
%\begin{eqnarray*}
% A^2(\tau)-2A(\tau) B(\tau) &\geq & \sum_{\lambda_j<0}  |\lambda_j|^2 \abs{s_j}^4 e^{-2\tau\lambda_j} -2 \sum_{\lambda_j<0,\lambda_i>0}  |\lambda_j| |\lambda_i| \abs{s_j}^2 \abs{s_i}^2 e^{-\tau(\lambda_j+\lambda_i)}\\
% &=& \sum_{\lambda_j<0} e^{-2\tau \lambda_j} \left( |\lambda_j|^2 |s_j|^4-2 \sum_{\lambda_i>0}  |\lambda_j|  |\lambda_i| |s_j|^2 |s_i|^2 e^{-\tau (\lambda_i+|\lambda_j|)} \right)\\
% &\geq & \sum_{\lambda_j<0} e^{-2\tau \lambda_j} \left( |\lambda_j|^2 |s_j|^4-2C_1^2 \sum_{\lambda_i>0}  |\lambda_j|  |\lambda_i| e^{-\tau (\lambda_i+|\lambda_j|)} \right)\\
 % &\geq & \sum_{\lambda_j<0} e^{-2\tau \lambda_j} \left( |\lambda_j|^2 |s_j|^4- o(1) \right).
 %\end{eqnarray*}
\noindent Combining all the above we get that for any $t\in \R$
 \begin{eqnarray*}
  \int_X u'(\tau) \omega_{u(t)}^n & \geq & \frac{1}{k} \frac{ \sum_{\lambda_j<0} e^{-\tau \lambda_j}  \left(\int_X  |\lambda_j| |s_j|^2 \omega_{u(t)}^n -h(\tau)\right) }{C_1 \left( \sum_{\lambda_\ell<0 }e^{-\tau \lambda_\ell}+1 \right)}.
    \end{eqnarray*}
It follows that $$\liminf_{\tau\rightarrow +\infty}   \int_X u'(\tau) \omega_{u(t)}^n  \geq  \frac{1}{C_1k} \inf_{\lambda_j<0}  \int_X |\lambda_j||s_j|^2 \omega_{u(t)}^n,$$
where the infimum is strictly positive since $\{s_j\}$ is an orthonormal basis, in particular $s_j$ can not be identically zero. 
By definition of liminf, there exists $\tau_0$ and $\alpha>0$ such that for any $s\in \R$, $\int_X u'(\tau)\, \omega_{u(t)}^n  \geq \alpha>0$ for $\tau \geq \tau_0$. The same arguments will give that $ \int_X u'(\tau)\, \omega_{u(s)}^n  \geq \beta>0$ for $\tau \geq \tau_0$. Up to relabelling, we can assume that $\alpha=\min (\alpha, \beta)$.
\smallskip

For $\tau\in \R^-$, the same arguments, switching the role of positive and negative eigenvalues, will ensure that there exists $\tau'_0, \alpha'>0$ such that $\int_X |u'(\tau)| \, \omega_{u(t)}^n+  \omega_{u(s)}^n \geq \alpha'>0$ for $\tau \leq -\tau'_0$. Up to relabelling again, we can assume that $\tau_0=\max (\tau_0, \tau_0')$ and $\alpha=\min (\alpha, \alpha')$. We stress that in what follows we need $\alpha>0$ which means that we need to restrict to the case $\lambda_1<0$ and $\lambda_{N_k}>0$.
\smallskip
Thanks to the fundamental theorem of calculus we know that for any $s<t \in \R$, $u(t)-u(s)=\int_s^tu'(\tau) d\tau$. Then, using Lemma \ref{Lemma:DarvasDistanceL2normComparison} and Jensen inequality we have
\begin{eqnarray*}
d_p(u(t), u(s))^p &\geq & C_p^{-1} \int_X \left|\int_s^t u'(\tau)  d\tau \right|^p\omega_{u(t)}^n+\omega_{u(s)}^n\\
&\geq & C_p^{-1}\left| \int_s^t \left(\int_X u'(\tau) \, \omega_{u(t)}^n+\omega_{u(s)}^n \right)d\tau \right|^p
\end{eqnarray*}
We now distinguish five cases: 
\begin{itemize}
\item[(i)] $s,t>\tau_0$, 
\item[(ii)] $s,t\in [-\tau_0,\tau_0]$,
\item[(iii)] $s\in [-\tau_0,\tau_0], t >\tau_0$,
\item[(iv)] $s<-\tau_0, t \in [-\tau_0,\tau_0]$,
\item[(v)] $s,t<-\tau_0$.
\end{itemize}
When $s,t>\tau_0$ 
$$d_p(u(t), u(s))^p \geq  C_p^{-1} (2\alpha)^p |t-s|^p, $$ hence
$$d_p(u(t), u(s)) \geq  \frac{\alpha}{C_p'} (t-s) = \frac{\alpha}{{C_p'}} \sqrt{\frac{N_k}{\sum_i \lambda_i^2}} \,\hat{d}(H(t), H(s)). $$
When $s,t\in [-\tau_0,\tau_0]$ then $t-s\leq 2\tau_0$, hence
$$ d_p(u(t), u(s)) \geq 0\geq (t-s)-2\tau_0= \sqrt{\frac{N_k}{\sum_i \lambda_i^2}} \,\hat{d}(H(t), H(s))-2\tau_0,$$
We now discuss the case (iii). We use the linearity of the integral and observe that 
$$
\int_s^{\tau_0} \left(\int_X u'(\tau) \, \omega_{u(t)}^n+\omega_{u(s)}^n \right)d\tau  + \int_{\tau_0}^t  \left(\int_X u'(\tau) \, \omega_{u(t)}^n+\omega_{u(s)}^n \right)d\tau
$$
\begin{eqnarray*}
&\geq & -2C_0 (\tau_0-s) +2\alpha (t-\tau_0)= 2\alpha (t-s) -(2\alpha+2C_0) (\tau_0-s)\\
&\geq & 2\alpha (t-s) -(2\alpha+2C_0) \tau_0,
\end{eqnarray*}
where the first inequality follows from the fact that $|u'(\tau)|\leq C_0$ on the compact interval $[-\tau_0, \tau_0]$, since $u'$ is a continuous function. 
We consider two sub-cases. Let $a_p=2^{p+1} (2^p-1)$.
When $\alpha(t-s) \leq a_p\tau_0(\alpha+C_0)$ we can conclude easily since
$$ d_p(u(t), u(s)) \geq 0\geq \alpha(t-s)- a_p\tau_0(\alpha+C_0) = \alpha\sqrt{\frac{N_k}{\sum_i \lambda_i^2}} \,\hat{d}(H(t), H(s))-a_p\tau_0(\alpha+C_0). $$
When $\alpha(t-s) > a_p\tau_0(\alpha+C_0)$, we observe that for each $j=0,\cdots, p-1$ we have $$\frac{c_{j,p}}{2}[\alpha (t-s)]^{p-j} > \binom{p}{j}[(\alpha+C_0) \tau_0]^{b-j},$$
where $c_{j,p}:=\frac{2^j}{2^p-1}$. Hence
\begin{eqnarray*}
d_p(u(t), u(s))^p &\geq &  C_p^{-1} ( \alpha (t-s) -(\alpha+C_0) \tau_0)^p \\
&\geq &C_p^{-1} \left( \sum_{j=0}^{p-1} c_{j,p} \alpha^p (t-s)^p+ \sum_{j=0}^{p-1} \binom{p}{j}(-1)^j \alpha^j (t-s)^j(\tau_0(\alpha+C_0))^j\right)\\
&\geq & {C'}_p^{-1}  \alpha^p (t-s)^p. 
\end{eqnarray*}
Thus $$d_p(u(t), u(s))\geq {C''}_p^{-1} \alpha\sqrt{\frac{N_k}{\sum_i \lambda_i^2}} \,\hat{d}(H(t), H(s)).$$
We can finally deduce that there exist positive constants $C_k=C_k(A, N_k,n, \alpha,p), D_k=D_k(\alpha,p, \tau_0)$ such that
$$d_p(u(t), u(s))\geq C_k \hat{d}(H(t), H(s)) -D_k.$$
The cases (iv) and (v) follow exactly as the cases in (iii) and (i), respectively. The proof is complete.
\end{proof}

\section{Quantization in the semi-positive case}
\label{Sec:PotentialsWithSemiPositiveForm}
Let $(X,\omega)$ be a compact K\"ahler manifold of complex dimension $n$ and $\theta$ be a real and closed $(1,1)$-form which is semi-positive and big.
In this section we describe an approach to give a \emph{quantization} result for the geodesics in the space 
$$\cH_\theta :=\{ u\in \PSH (X, \theta)\cap L^\infty(X) \cap C^\infty(\Amp(\theta )) \; :\; \theta_u=\theta+dd^c u>0 \; {\rm{in}} \; \Amp(\theta )\}.$$

It was proved in \cite{di2018geometry} that $\cH_\theta$ is a metric space. For $\varepsilon>0$ let $\omega_\varepsilon:=\theta+\varepsilon \omega>0$, and $d_{p, \omega_\varepsilon}$ denote the metric on $\overline{\cH_{\omega_\varepsilon}}\supset \cH_\theta$.
 In \cite{di2018geometry}, the authors proved that for any $u_0, u_1\in \cH_\theta$, $\lim_{\varepsilon\rightarrow} d_{p,\omega_\varepsilon}(u_0, u_1)$ exists and defines a metric, denoted by $d_p$ as well. 
 Moreover, for $t=0,1$, one has an explicit expression of the distance:
 $$d_p(u_0, u_1)= \left(\frac{1}{V}\int_X |\dot{u}_t|^p \theta^n_{u_t}\right)^{1/p},$$
 where $t\rightarrow u_t$ is the geodesic joining $u_0$ and $u_1$ and $V:=\int_X \theta^n > 0$.\\
 We recall that given $u_0, u_1 \in \Hspace_\theta$, the \emph{weak geodesic} from $u_0$ to $u_1$ is the path $t\mapsto u_t$ defined as follows: for any $x\in X$ and $z\in S := \left\lbrace 0<\Rea(z)<1 \right\rbrace$, setting $t= \Rea (z) \in [0,1]$, we define
\begin{equation}
\label{eq:DefGeodesicTheta}
u_t(x) := U(x, \Rea(z)),
\end{equation}
where $U$ is defined as the supremum:
\begin{align}
\label{eq:DefGeodesicSupTheta}
U(x,z) = \sup \left\lbrace V \in \PSH(X\times S,  \pi_1^*{\theta}) \, \Bigg\vert \, \limsup_{\Rea(z) \rightarrow 0,1}  V(x,z) \leq u_{0,1}(x) \right\rbrace.
\end{align}
Here $\pi_1: X\times S\rightarrow X$ is the projection on the first factor.

\begin{Prop}[{\cite[Proposition 1.4]{di2018geometry}}]
\label{Prop:RegularityWeakGeodesicTheta}
Let $U$ be defined as in \eqref{eq:DefGeodesicSupTheta}. Then $U\in \PSH(X\times S, \pi_1^*{\theta})\cap L^\infty (X\times S)$. In addition $U$ satisfies, for any $(x,z,z') \in X \times S \times S$:
$$
\abs{U(x,z) - U(x,z')} \leq \left( \sup_X(u_0 - u_1) \right) \abs{\Rea(z) - \Rea(z')}.
$$
Finally, $U$ is continuous in $\Amp(\theta ) \times \overline{S}$ and satisfies:
\begin{equation}
\label{eq:MA-equation-Geodesic-theta}
\begin{cases}
U\vert_{\left\lbrace\Rea(z) = 0\right\rbrace} = u_0 \; ; \quad U\vert_{\left\lbrace\Rea(z) = 1\right\rbrace} = u_1 \\
\left(\pi_1^*\theta + dd^c_{(x,z)} U\right)^{n+1} = 0,
\end{cases}
\end{equation}
and it is the unique bounded $\theta$-plurisubharmonic solution to the system of equations~\eqref{eq:MA-equation-Geodesic-theta}.
\end{Prop}

In particular, for any $x\in X$, the map $t\mapsto u_t(x)$ defined by \eqref{eq:DefGeodesicTheta} is Lipschitz continuous. The Lipschitz constant being $ \sup_X(u_0 - u_1)$, it is uniform in $x\in X$.

\medskip
Moreover, the geodesic $ U\in \PSH(X\times S, \pi_1^*{\theta})\cap L^\infty (X\times S)$ is the decreasing limit of geodesics (joining $u_0, u_1$) with respect to $\omega_\varepsilon$.
More precisely, if we consider
$$
U^\varepsilon := \sup \left\lbrace V \in \PSH(X\times S,  \pi_1^*{\omega_\varepsilon}) \; \vert \; \limsup_{\Rea(z) \rightarrow 0,1}  V(x,z) \leq u_{0,1}(x) \right\rbrace,
$$
then we have:

\begin{Prop}
The sequence $\{U^\varepsilon\}$ is decreasing. In addition:
\begin{equation}
U^\varepsilon \underset{\varepsilon\rightarrow 0}{\longrightarrow} U \quad \text{uniformly on } K\times \overline{S},
\end{equation}
for any compact subset $K \subset \Amp( \theta )$.
\end{Prop}

\begin{proof}
The decreasing convergence on $X\times \overline{S}$ follows from \cite[Proposition~1.6]{di2018geometry}. Let us prove that the convergence is uniform on any compact set $K\subset \Amp( \theta )$. Proposition \ref{Prop:RegularityWeakGeodesicTheta} asserts that $U$ and $U^\varepsilon$ are continuous on $K\times \overline{S}$. Since the sequence $\{U^\varepsilon\}$ is decreasingly converging to $U$, Dini's theorem gives uniform convergence on $K\times \overline{S}$.
\end{proof}

As before, for any $t\in [0,1]$ and $z \in \overline{S}$ such that $\Rea(z) = t$ we set:
\begin{equation}
\label{eq:DefPathOfStep1}
u_t^\varepsilon(\cdot) := U^\varepsilon(\cdot, z).
\end{equation}
It follows from the previous proposition that for any compact subset $K\subset \Amp( \theta)$, we have the decreasing convergence of the weak $\omega_\varepsilon$-geodesics $t\mapsto u_t^\varepsilon$ towards the weak $\theta$-geodesic $t\mapsto u_t$:
\begin{equation}
\label{eq:QuantizationStep1}
\sup_{(x,t) \in K\times \overline{S}} \abs{u^\varepsilon_t(x) - u_t(x)}     \underset{\varepsilon\rightarrow 0}{\longrightarrow} 0.
\end{equation}

We now assume that the forms $\theta$ and $\omega$ are both integral. Recall that a $2$-form is said to be \emph{integral} if its cohomology class lies in the image $H^2(X,\Z)_b$ of the natural map $H^2(M,\Z) \rightarrow H^2_{dR}(M)$. We then give an approximation of the weak geodesics segments defined in \eqref{eq:DefGeodesicTheta} by paths coming from a finite dimensional space.
\smallskip

Let $t\mapsto u_t^\ell$ be the path defined in \eqref{eq:DefPathOfStep1} with $\varepsilon=1/\ell$. Since $\omega_\ell$ is Kähler, we can consider smooth approximations as in \cite[Theorem 1]{blocki_regularization_2007}:
let $u_{0,j}$ (resp. $u_{1,j}$) be a decreasing sequence of smooth strictly $\omega_\ell$-plurisubharmonic functions decreasing towards $u_0$ (resp. $u_1$)
and let $u^\ell_{t,j}$ be the weak geodesic between $u_{0,j}$ and $u_{1,j}$. Recall that the weak geodesics $u_t^\ell$ and $u_{t,j}^\ell$ are defined so that $u^\ell_{\Rea(z)}(x)=U^\ell(x,z)$ and $u^\ell_{\Rea(z),j}(x)=U^\ell_j(x,z)$ where
$$
U^\ell := \sup \left\lbrace V \in \PSH(X\times S,  \pi_1^*{\omega_\ell}) \; \vert \; \limsup_{\Rea(z) \rightarrow 0,1}  V(x,z) \leq u_{0,1}(x) \right\rbrace
$$

$$
U^\ell_j := \sup \left\lbrace V \in \PSH(X\times S,   \pi_1^*{\omega_\ell}) \; \Bigg\vert \;
\begin{cases}
\limsup_{\Rea(z) \rightarrow 0}  V(x,z) \leq u_{0,j}(x) \\
\limsup_{\Rea(z) \rightarrow 1}  V(x,z) \leq u_{1,j}(x)
\end{cases} \right\rbrace.
$$

We then have the following:

\begin{Prop}
\label{Prop:QuantizationStep2}
Fix $\ell \in \N$. The sequence $\left( U^\ell_j \right)_{j\in \N}$ decreases to $U_\ell$ uniformly in $X\times \overline{S}$,
\begin{equation}
\label{eq:QuantizationStep2}
\sup_{(x,t) \in X \times [0,1]} \left\vert u_{t,j}^\ell(x)- u_t^\ell(x) \right\vert \underset{j\rightarrow +\infty}{\longrightarrow} 0.
\end{equation}
\end{Prop}
In particular we can choose $j=j(\ell)$ such that 
\begin{equation}
\label{eq j conv}
\sup_{(x,t) \in X \times [0,1]} \left\vert u_{t,j}^\ell(x)- u_t^\ell(x) \right\vert \leq \frac{1}{\ell}.
\end{equation}
\smallskip

\noindent The proof of such a result is somehow classical. We give details for reader's convenience.
\begin{proof}
It follows from the comparison principle that the sequence $\left(U_j^\ell\right)_{j\in\N}$ is decreasing. Let us show that this sequence converges towards $U^\ell$. First, since $U^\ell$ is a candidate in the supremum defining $U_j^\ell$ (Proposition \ref{Prop:RegularityWeakGeodesicTheta}), we have
$
U^\ell \leq  U_j^\ell,
$
for any $j$. Passing to the limit for $j\rightarrow +\infty$ we get $U^\ell\leq W^\ell:=\lim_{j\rightarrow \infty} U_j^\ell $. In particular $W^\ell$ is a genuine $\pi_1^*\omega_\ell$-psh function on $M$. Hence to infer that 
$W^\ell\geq U^\ell$, we only need to check that $W^\ell$ has the right boundary conditions.
Since $t\mapsto u_{t,j}^\ell$ is convex and $t\mapsto t u_{1,j} + (1-t) u_{0,j}$ is affine with same endpoints we have that for all $t \in [0,1]$,
$$
\quad u_{t,j}^\ell \leq t u_{1,j} + (1-t) u_{0,j}.
$$
Letting $j\rightarrow +\infty$ and taking $t=0,1$ give:
$$
\lim_{j\rightarrow \infty} u_{0,j}^\ell \leq u_0 \; ; \qquad \lim_{j\rightarrow \infty} u_{1,j}^\ell \leq u_1.
$$
In other words:
$$
\lim_{j\rightarrow \infty} U_j^\ell \vert_{\left\lbrace\Rea(z) = 0\right\rbrace} \leq u_0 \; ; \qquad
\lim_{j\rightarrow \infty} U_j^\ell \vert_{\left\lbrace\Rea(z) = 1\right\rbrace} \leq u_1.
$$
This shows that $W^\ell$ is candidate for the $\sup$ defining~$U^\ell$. Moreover, the convergence is uniform thanks to Dini's Theorem: indeed, $\left( U_j^\ell \right)_{j\in\N}$ is a decreasing sequence and, since $\omega_\ell$ is Kähler, Proposition \ref{Prop:RegularityWeakGeodesicTheta} asserts that all the functions $U_j^\ell$ and $U^\ell$ are continuous on $X \times \overline{S}$.
\end{proof}

\begin{Th}
\label{Th:QuantizationGeodesicsSemiPositive}
Let $u_0, u_1 \in \cH_\theta$ and $u_t$ be the weak geodesic joining them. There exists a sequence of path $(\hat{u}_{t,\ell})_{\ell\in\N}$ in the image of $\frac{1}{\ell}\FS_k$, where $k$ depends on $\ell$, such that for any compact $K\subset \Amp(\theta)$, we have:
$$
\sup_{(x,t) \in K\times [0,1]} \left\vert \hat{u}_{t,\ell}(x) - u_t(x)\right\vert  \rightarrow_{\ell \rightarrow \infty} 0.
$$
\end{Th}

\begin{proof}
We apply the quantization process to $t\mapsto u_{t,j}^\ell$, with $L$ being an ample line bundle whose curvature is $\ell \theta+\omega$. Hence
\begin{equation}
\label{eq:SchemaHilbFS}
\begin{tikzcd}
\Hspace_{\ell\theta + \omega}
\arrow[r, yshift=0.7ex, "\Hilb_k"]
& \hat{\cH}(L^k)
\arrow[l, yshift=-0.7ex, "\FS_k"].
\end{tikzcd}
\end{equation}
By construction
$\ell u_{0,j}^\ell, \ell u_{1,j}^\ell \in  \cH_{\ell\theta + \omega}$. 
We then consider the geodesic $ [0,1] \ni t\mapsto H_t^{k,j,\ell} \in \hat{\cH}(L^k)$ that joins $\Hilb_k(\ell u_{0,j}^\ell)$ to $\Hilb_k(\ell u_{1,j}^\ell)$. Theorem \ref{Th:GeometricQuantizationKahler} ensures the existence of a constant $C_{\ell,j}$ such that

\begin{equation}
\label{eq:QuantizationStep3BIS}
\sup_{(x,t) \in X\times [0,1]} \abs{ \frac{1}{\ell}\FS_k\left(H^{k,j,\ell}_t\right)(x) - u_{t,j}^\ell(x) } \leq \frac{C_{\ell,j}}{\ell} \frac{\log k}{k}.
\end{equation}
We now choose $j=j(\ell)$ so that \eqref{eq j conv} holds, and since, for $\ell\in \mathbb{N}$ fixed, $
\lim_{k\rightarrow \infty}\frac{C_{\ell,j(\ell)}}{\ell} \frac{\log k}{k} = 0,
$
we can choose $k=k(\ell)$ so that:
\begin{equation}
\label{eq:ChoiceOf_k}
\sup_{(x,t) \in X\times [0,1]} \left|\frac{1}{\ell} \FS_{k(\ell)} \left(H^{k(\ell),j(\ell),\ell}_t\right)(x) - u_{t,j(\ell)}^{\ell}(x) \right| \leq \frac{1}{\ell}
\end{equation}
We claim that, choosing $j$ and $k$ as above, 
$\hat{u}_{t,\ell} := \frac{1}{\ell}\FS_{k}\left(H^{k,j,\ell}_t\right)
$ is a \emph{quantized path} for $u_t$. Indeed for any $K$ compact subset of $\Amp(\theta )$, 
\begin{align}
\sup_{(x,t) \in K\times [0,1]} \left\vert \hat{u}_{t,\ell}(x) - u_t(x)\right\vert
&\leq \sup_{(x,t) \in K\times [0,1]} \abs{ \frac{1}{\ell}\FS_k\left(H^{k,j,\ell}_t\right)(x) - u_{t,j}^\ell(x) } \\
& \qquad + \sup_{(x,t) \in K\times [0,1]} \left\vert u_{t,j}^\ell(x)- u_t^\ell(x) \right\vert \\
& \qquad + \sup_{(x,t) \in K\times [0,1]} \abs{ u_t^\ell (x) - u_t(x)} \\[1em]
&\leq \frac{1}{\ell} + \frac{1}{\ell} + \sup_{(x,t) \in K\times [0,1]} \abs{ u_t^\ell (x) - u_t(x)}.
\label{eq:CalculFinalGeodesicApproximationSemiPositif}
\end{align}
The conclusion follows from \eqref{eq:QuantizationStep1}.
\definecolor{ttzzqq}{rgb}{0.2,0.6,0}
\definecolor{qqqqff}{rgb}{0,0,1}

\begin{figure}[h]
\centering
\begin{minipage}{0.4\textwidth}
\begin{figure}[H]
\scalebox{0.65}{
\centering
\begin{tikzpicture}% Figure  approximation DiN-G
\draw [shift={(0,-0.44)}] plot[domain=3.62:5.8,variable=\t]({1*3.38*cos(\t r)+0*3.38*sin(\t r)},{0*3.38*cos(\t r)+1*3.38*sin(\t r)});
\draw [shift={(0,2)},color=ttzzqq]  plot[domain=4.07:5.36,variable=\t]({1*5*cos(\t r)+0*5*sin(\t r)},{0*5*cos(\t r)+1*5*sin(\t r)});
\draw [shift={(0,5)},color=ttzzqq]  plot[domain=4.31:5.12,variable=\t]({1*7.62*cos(\t r)+0*7.62*sin(\t r)},{0*7.62*cos(\t r)+1*7.62*sin(\t r)});
\draw [->, color=ttzzqq] (0,-3.15) -- (0,-3.65);
\draw [dotted, color=ttzzqq] (0,-2.7) -- (0,-3);
\draw (-3,-2) node[left] {$u_0$};
\draw (3,-2) node[right] {$u_1$};
\draw (0.5,-4) node[right] {$u_t$};
\draw (0.5,-2.2) node[right, color=ttzzqq] {$u_t^\ell$};
\end{tikzpicture}
}
\caption*{Step 1}
\end{figure}

\begin{figure}[H]
\scalebox{0.65}{
\centering
\begin{tikzpicture}%Figure Demailly approx at level n
\draw [shift={(0,-0.44)}] plot[domain=3.62:5.8,variable=\t]({1*3.38*cos(\t r)+0*3.38*sin(\t r)},{0*3.38*cos(\t r)+1*3.38*sin(\t r)});
\draw [shift={(0,2)},color=ttzzqq]  plot[domain=4.07:5.36,variable=\t]({1*5*cos(\t r)+0*5*sin(\t r)},{0*5*cos(\t r)+1*5*sin(\t r)});
\draw [shift={(0,2.5)},color=qqqqff]  plot[domain=4.07:5.36,variable=\t]({1*5*cos(\t r)+0*5*sin(\t r)},{0*5*cos(\t r)+1*5*sin(\t r)});
\draw [shift={(0,4)},color=qqqqff]  plot[domain=4.07:5.36,variable=\t]({1*5*cos(\t r)+0*5*sin(\t r)},{0*5*cos(\t r)+1*5*sin(\t r)});
\draw [->, color=qqqqff] (0,-2.64) -- (0,-2.9);
\draw [dotted, color=qqqqff] (0,-1.5) -- (0,-2);
\draw [dotted] (-3.4,-0.5) -- (-3.4,-1.5);
\draw [dotted] (3.4,-0.5) -- (3.4,-1.5);
\draw (-3,-2) node[left] {$u_0$};
\draw (-3,0) node[left] {$u_{0,j}$};
\draw (3,-2) node[right] {$u_1$};
\draw (3,0) node[right] {$u_{1,j}$};
\draw (0.5,-4) node[right] {$u_t$};
\draw (0.5,-3.2) node[right, color=ttzzqq] {$u_t^{\ell}$};
\draw (0.5,-0.5) node[right, color=qqqqff] {$u_{t,j}$};
\end{tikzpicture}
}
\caption*{Step 2}
\end{figure}
\end{minipage}
\hspace{0.02\textwidth}
\begin{minipage}{0.4\textwidth}
\begin{figure}[H]
\scalebox{0.65}{
\centering
\begin{tikzpicture}%Quantization
\draw [shift={(0,-0.44)}] plot[domain=3.62:5.8,variable=\t]({1*3.38*cos(\t r)+0*3.38*sin(\t r)},{0*3.38*cos(\t r)+1*3.38*sin(\t r)});
\draw [shift={(0,2)},color=ttzzqq]  plot[domain=4.07:5.36,variable=\t]({1*5*cos(\t r)+0*5*sin(\t r)},{0*5*cos(\t r)+1*5*sin(\t r)});

\draw [shift={(0,1.59)},color=red]  plot[domain=3.94:5.48,variable=\t]({1*4.31*cos(\t r)+0*4.31*sin(\t r)},{0*4.31*cos(\t r)+1*4.31*sin(\t r)});
\draw [shift={(0,3.5)},color=red]  plot[domain=3.61:5.82,variable=\t]({1*3.35*cos(\t r)+0*3.35*sin(\t r)},{0*3.35*cos(\t r)+1*3.35*sin(\t r)});

\draw [shift={(0,2.5)},color=qqqqff]  plot[domain=4.07:5.36,variable=\t]({1*5*cos(\t r)+0*5*sin(\t r)},{0*5*cos(\t r)+1*5*sin(\t r)});
\draw [shift={(0,4)},color=qqqqff]  plot[domain=4.07:5.36,variable=\t]({1*5*cos(\t r)+0*5*sin(\t r)},{0*5*cos(\t r)+1*5*sin(\t r)});

\draw [shift={(0,6)},color=qqqqff]  plot[domain=4.07:5.36,variable=\t]({1*5*cos(\t r)+0*5*sin(\t r)},{0*5*cos(\t r)+1*5*sin(\t r)});
\draw (-3,-2) node[left] {$u_0$};
\draw (-3,0) node[left] {$u_{0,j}$};
\draw (3,-2) node[right] {$u_1$};
\draw (3,0) node[right] {$u_{1,j}$};
\draw (-3,-1.4) node[left] {$u_{0,j(\ell+1)}$};
\draw (3,-1.4) node[right] {$u_{1,j(\ell+1)}$};
\draw (-3,2) node[left] {$u_{0,j(\ell)}$};
\draw (3,2) node[right] {$u_{1,j(\ell)}$};
\draw (0.5,-4) node[right] {$u_t$};
\draw (-1,-0.7) node[right, color=qqqqff] {$u_{t,j}$};
\draw (-1,1.3) node[right, color=qqqqff] {$u_{t,j(\ell)}$};
\draw (0.5,0.6) node[right, color=red] {$\hat{u}_{t,\ell}$};
\draw (0.5,-2) node[right, color=red] {$\hat{u}_{t,\ell+1}$};
\draw [->, color = red] (0,-2.8) -- (0,-3.5);
\end{tikzpicture}
}
\caption*{$\qquad\quad$ Step 3}
\end{figure}
\end{minipage}
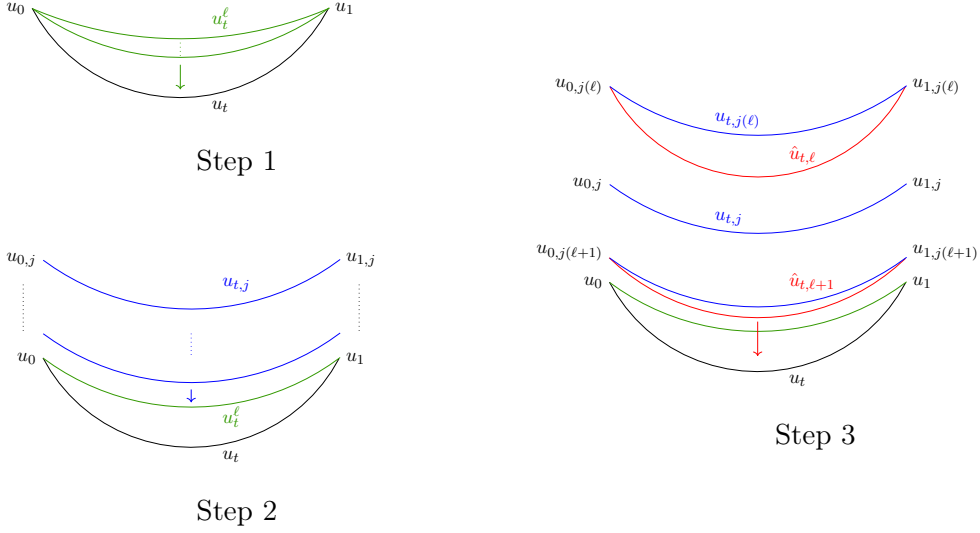
\caption{Quantization Process for Geodesics}
\label{Fig:Quantization Process}
\end{figure}
\end{proof}

\smallskip

\section{Quantization in Sasaki geometry}
\label{sec:QuantizationSasaki}

In this section, we explain how to adapt the quantization of geodesics done in Section \ref{Sec:PotentialsWithSemiPositiveForm} to the Sasaki setting. Let $(M,\mathcal{S})$ be a Sasaki manifold. We assume $(M,\mathcal{S})$ to be regular or quasi-regular.

If $(M,\mathcal{S})$ is a regular Sasaki manifold over the Kähler manifold $(X,\omega)$, the space of potentials $\cH(M,\mathcal{S})$ is homothetic to the space of Kähler potentials $\cH_\omega$ (see \cite{franzinetti2021}). Thus, the quantization of $\cH(M,\mathcal{S})$ reduces to the quantization of $\cH_\omega$.

\medskip

We then consider the case $(M,\mathcal{S})$ to be a quasi-regular Sasaki manifold over the Kähler orbifold $(X,\omega)$. Let $\pi : M \rightarrow X$ be the quotient map and denote by $M^{\text{reg}} := \pi^{-1} (X^{\text{reg}}),$ the set of points in $M$ whose orbit under the Reeb vector field is regular. By \cite[Theorem 7.1.3 (iii)]{boyer_sasakian_2007} we can assume $\omega$ to be an integral $2$-form. 
We consider a resolution of singularities $\tilde{X}$ of the Kähler orbifold $(X,\omega)$:
\begin{equation}
\begin{tikzcd}
 & \tilde{X} \ar[d, "\rho"] \\
M \ar[r, "\pi"] &  X.
\end{tikzcd}
\end{equation}
Here $\rho$ is a holomorphic map which induces a biholomorphism between $\rho^{-1}(X^{\text{reg}})$ and $X^{reg}$. By construction $\tilde{X}$ is a K\"ahler manifold endowed with a Kähler form $\tilde{\omega}$, while $\theta := \rho^*\omega$ is a real closed $(1,1)$-form which is big and semi-positive. Observe that, since $X$ is projective, both $\tilde{\omega}$ and $\theta$ are integral forms. The above diagram induces:
\begin{equation}
\begin{tikzcd}
 & \cH_\theta \\
\cH(M,\mathcal{S}) \ar[r, bend right = 50 , "\mathrm{p} "] & \cH_\omega \ar[u, "\rho^*"]  \ar[l, "\mathrm{f}"'].
\end{tikzcd}
\end{equation}
Here $\mathrm{f}$ is the correspondence between the set of basic functions on $M$ and the set of smooth functions on $X$. The map $\mathrm{f}$ is a $1$-to-$1$ correspondence, and $\mathrm{p}$ is its inverse . Also, the set $\cH_\omega$ is the set of smooth functions $u$ on the orbifold $X$ such that $\omega+dd^c u > 0$ in the sense that the $2$-form $\omega+dd^c u$, computed on each local lifts of orbifold charts, is a positive $(1,1)$-form.

The quantization result in this setting states as follows:
\begin{Th}
\label{Th:QuantizGeodesicSasaki}
Let $u_0, u_1 \in \cH(M,\mathcal{S})$ and let $u_t$ be the weak geodesic joining $u_0$ and $u_1$. There exists a sequence of path $(\hat{u}_{t,\ell})_{\ell\in\N}$, which are image of geodesics in a finite dimensional space, such that for any compact set $K\subset M^{\text{reg}}$, we have the following uniform convergence:
$$
\sup_{(x,t) \in K\times [0,1]} \abs{ \hat{u}_{t,\ell}(x) - u_t(x)} \longrightarrow_{\lim_{\ell\rightarrow \infty} } 0.
$$
\end{Th}

\begin{proof}[Proof of Theorem \ref{Th:QuantizGeodesicSasaki}]
We start observing that
$
v_t:=\rho^* \left( \mathrm{p} (u_t) \right)
$
is the weak geodesic in $\cH_\theta$ joining $ \rho^* \left( \mathrm{p} (u_0)\right)$ and $ \rho^* \left( \mathrm{p} (u_1)\right)$.
We then apply the quantization process (Theorem \ref{Th:QuantizationGeodesicsSemiPositive}) to $v_t$. This gives a sequence of paths defined on $\tilde{X}$:
$$
t \mapsto \hat{v}_{t,\ell} \; \in \cH_{\theta + \frac{1}{\ell}\tilde{\omega}}, \qquad \ell\in\N.
$$
Moreover, the sequence $\hat{v}_{t,\ell}$ converges uniformly (as $\ell\rightarrow +\infty$) to $v_t $ on each compact set $K\subset \Amp(\theta)=\rho^{-1}(X^{\text{reg}})$.
Since $\rho$ is a biholomorphism on $\rho^{-1}(X^{\text{reg}})$, we infer that on any compact subset of $X^{\text{reg}}$, the following uniform convergence holds
$$
\rho_* \, \hat{v}_{t,\ell} \; \longrightarrow v_t, \quad {\rm as}\;\, \ell\rightarrow +\infty.
$$
At the level of the Sasaki manifold $(M,\mathcal{S})$, the above ensures uniform convergence $$
\mathrm{f} (\rho_* \, \hat{v}_{t,\ell}) \longrightarrow u_t
$$ on any compact subset of $M^{\text{reg}}$.
This ends the proof with 
$
\hat{u}_{t,\ell} := \mathrm{f} (\rho_* \, \hat{v}_{t,\ell}).
$
\end{proof}

\begin{footnotesize}
\addcontentsline{toc}{section}{Bibliography}
\bibliographystyle{abbrv}
\bibliography{biblio}
\end{footnotesize}
\end{document}